\documentclass[a4]{amsart}

\input xypic
\usepackage{amssymb}

\newtheorem{theorem}{Theorem}[section]
\newtheorem{proposition}[theorem]{Proposition}
\newtheorem{lemma}[theorem]{Lemma}
\newtheorem{corollary}[theorem]{Corollary}

\theoremstyle{definition}

\newtheorem{remark}[theorem]{Remark}

\newcommand{\g}{\ensuremath{\mathcal{G}}} 
\newcommand{\gk}{\ensuremath{\mathcal{G}_{k}}} 
\newcommand{\gbark}{\ensuremath{\mathcal{G}_{\bar{k}}}}

\newcommand{\map}{\ensuremath{\mbox{Map}}} 
\newcommand{\mapstar}{\ensuremath{\mbox{Map}^{\ast}}} 

\newcounter{bean}
\newenvironment{letterlist}{\begin{list}{\rm ({\alph{bean}})}
      {\usecounter{bean}\setlength{\rightmargin}{\leftmargin}}}
      {\end{list}}

\newcommand{\namedright}[3]{\ensuremath{#1\stackrel{#2}
 {\longrightarrow}#3}}
\newcommand{\nameddright}[5]{\ensuremath{#1\stackrel{#2}
 {\longrightarrow}#3\stackrel{#4}{\longrightarrow}#5}}
\newcommand{\namedddright}[7]{\ensuremath{#1\stackrel{#2}
 {\longrightarrow}#3\stackrel{#4}{\longrightarrow}#5
  \stackrel{#6}{\longrightarrow}#7}}

\newcommand{\larrow}{\relbar\!\!\relbar\!\!\rightarrow}
\newcommand{\llarrow}{\relbar\!\!\relbar\!\!\larrow}

\newcommand{\llnamedright}[3]{\ensuremath{#1\stackrel{#2}
 {\llarrow}#3}}

\newcommand{\qqed}{\hfill\Box}

\newcommand{\zmodp}{\ensuremath{\mathbb{Z}\mathit{/p}\mathbb{Z}}}

\begin{document}


\title{The homotopy types of $U(n)$-gauge groups over lens spaces} 
\author{Ingrid Membrillo-Solis} 
\address{Mathematical Sciences, University
         of Southampton, Southampton SO17 1BJ, United Kingdom}
\email{I.Membrillo-Solis@soton.ac.uk} 
\author{Stephen Theriault}
\address{Mathematical Sciences, University
         of Southampton, Southampton SO17 1BJ, United Kingdom}
\email{S.D.Theriault@soton.ac.uk} 



\begin{abstract} 
We analyse the homotopy types of gauge groups for principal $U(n)$-bundles 
over lens spaces. 
\end{abstract}

\maketitle

\section{Introduction}
\label{sec:intro} 

Let $G$ be a simple, compact Lie group and let 
\(\namedright{P}{}{M}\) 
be a principal $G$-bundle. The \emph{gauge group} of this bundle 
is the group of $G$-equivariant automorphisms of $P$ that fix $M$. 
There has been a considerable amount of work recently in trying to 
understand the homotopy types of gauge groups that arise in physical 
or geometric contexts. Most work to date has concentrated on $M$ 
being a simply-connected four-manifold when $G$ is simply-connected 
or $M$ being an orientable surface when $G=U(n)$. 

In this paper we turn our attention to the case when $M$ is a $3$-manifold. 
If $G$ is simply-connected then $[M,BG]\cong 0$, implying that the only 
principal $G$-bundle is the trivial bundle, which has the trivial gauge group. 
We consider instead the more topologically intricate case when $G=U(n)$ 
and $M$ is a lens space, for then $[M,BU(n)]\not\cong 0$. 

Let $p$ and $q$ be coprime integers. The \emph{lens space} 
$L(p,q)$ is the orbit space $S^{3}/(\mathbb{Z}/s\mathbb{Z})$, 
where the action of $\mathbb{Z}/p\mathbb{Z}$ on $S^{3}$ is given by 
$(z_{0},z_{1})\longrightarrow (e^{2\pi i/p}z_{0},e^{2\pi iq/p}z_{1})$. 
For $n\geq 1$ and $p\geq 2$, let 
\(\underline{p}\colon\namedright{S^{n}}{}{S^{n}}\) 
be the map of degree $p$ and let $P^{n+1}(p)$ be its homotopy cofibre. 
The space $P^{n+1}(p)$ is the $(n+1)$-dimensional mod-$p$ 
\emph{Moore space}. As a $CW$-complex, $L(p,q)\simeq P^{2}(p)\cup e^{3}$. 

The analysis of gauge groups of principal $U(n)$-bundles over $L(p,q)$ is 
necessarily delicate for two reasons. First, the isomorphism classes of principal 
bundles is determined by $[L(p,q),BU(n)]$, and this set is determined by  
$[P^{2}(p),BU(n)]$ rather than $[S^{3},BU(n)]$. This is in contrast to the 
case when $G$ is simply-connected and $M$ is a simply-connected four-manifold 
or when $G=U(n)$ and~$M$ is an orientable surface; in both cases 
$[M,BG]$ is determined by the top cell of $M$ and this leads to certain homotopy 
fibrations being more easily compared. Second, typically localization techniques 
are used to work one prime at a time, allowing for easier progress. However, 
as $L(p,q)$ may not be nilpotent, localization techniques may be problematic, so we  
approach the problem without localization. The strategy and methods used should 
be applicable to other cases as well. 

Our main result is the following. As will be shown, there are isomorphisms 
$[L(p,q),BU(n)]\cong [P^{2}(p),BU(n)]\cong\mathbb{Z}/p\mathbb{Z}$. For 
$k\in\mathbb{Z}/p\mathbb{Z}$, let $\gk(L(p,q))$ and $\gk(P^{2}(p))$ be the 
gauge groups of the principal $U(n)$-bundles over $L(p,q)$ and $P^{2}(p)$ 
respectively with first Chern class $k$. For integers $a,b$, let $(a,b)$ be 
their greatest common divisor. 

\begin{theorem} 
   \label{main} 
   The following hold: 
   \begin{letterlist} 
      \item if $2\,|\, p$ then there is a homotopy equivalence 
               \[\Omega^{2}\gk(L(p,q))\simeq\Omega^{2}\gk(P^{2}(p))\times\Omega^{5} U(n);\] 
     \item if $k\equiv u\ell\bmod{n}$ where $(u,np)=1$ then there is a homotopy equivalence 
              \[\gk(P^{2}(p))\simeq\g_{\ell}(P^{2}(p)).\] 
   \end{letterlist} 
\end{theorem} 

Theorem~\ref{main}~(a) implies that if $2\mid p$ then the higher homotopy groups of $\gk(L(p,q))$ 
are determined by $\gk(P^{2}(p))$ and $U(n)$, and part~(b) implies that if in addition 
$k\equiv u\ell\bmod{n}$ for $(u,np)=1$ then $\Omega^{2}\gk(L(p,q))\simeq\Omega^{2}\g_{\ell}(L(p,q))$, 
even though the underlying principal $U(n)$-bundles might be distinct.

\section{Isomorphism classes of bundles and components of mapping spaces} 
\label{sec:bundles} 

As a $CW$-complex $L(p,q)\simeq P^{2}(p)\cup e^{3}$, so there is a homotopy cofibration 
\[\nameddright{S^{2}}{f}{P^{2}(p)}{i}{L(p,q)}\] 
where $f$ attaches the top cell to $L(p,q)$ and $i$ is the inclusion of the $2$-skeleton. Let 
\(\pi\colon\namedright{P^{2}(p)}{}{S^{2}}\) 
be the pinch map to the top cell. Let $j$ be the composite of inclusions 
\(j\colon\nameddright{S^{1}}{}{P^{2}(p)}{i}{L(p,q)}\). 
Then there is a homotopy cofibration diagram 
\begin{equation} 
  \label{Lpo} 
  \diagram 
      & S^{1}\rdouble\dto & S^{1}\dto^{j} & \\ 
      S^{2}\rto^-{f}\ddouble & P^{2}(p)\rto^-{i}\dto^{\pi} & L(p,q)\rto^-{g}\dto^{h} & S^{3}\ddouble \\ 
      S^{2}\rto^-{\pi\circ f} & S^{2}\rto & C\rto & S^{3} 
  \enddiagram 
\end{equation}  
that defines the space $C$ and the maps $g$ and $h$. 

\begin{lemma} 
   \label{Lpofacts} 
   The map $\pi\circ f$ is null homotopic and there is a homotopy equivalence 
   $C\simeq S^{2}\vee S^{3}$.   
\end{lemma} 

\begin{proof} 
The degree of $\pi\circ f$ is detected by the Bockstein in the homology of $C$, but this 
Bockstein is zero since the corresponding Bockstein for $L(p,q)$ is zero.  
Therefore $\pi\circ f\simeq\ast$, implying that $C\simeq S^{2}\vee S^{3}$. 
\end{proof} 

Let $\overline{\pi}$ be the composite 
\[\overline{\pi}\colon\namedddright{L(p,q)}{h}{C}{\simeq}{S^{2}\vee S^{3}}{}{S^{2}}\] 
where the right map collapses $S^{3}$ to a point. Lemma~\ref{Lpofacts} immediately 
implies the following. 

\begin{corollary} 
   \label{pinchext} 
   The pinch map 
   \(\namedright{P^{2}(p)}{\pi}{S^{2}}\) 
   extends across $i$ to a map 
   \(\overline{\pi}\colon\namedright{L(p,q)}{}{S^{2}}\).~$\qqed$  
\end{corollary} 

If $n=1$ then $BU(1)$ is the Eilenberg-Mac Lane space $K(\mathbb{Z},2)$ 
and for any $CW$-complex $X$ the set $[X,BU(1)]$ has a group structure. If $n>1$ 
then the standard inclusion 
\(\namedright{BU(n)}{}{BU(\infty)}\) 
has homotopy fibre $U(\infty)/U(n)$ which is $2n$-connected. Thus if $X$ 
is a $CW$-complex of dimension~$\leq 2n$ then there is an isomorphism 
$[X,BU(n)]\cong [X,BU(\infty)]$. In particular, as $BU(\infty)$ is an infinite loop 
space, $[X,BU(n)]$ has a group structure. In our case, each space in the 
homotopy cofibration sequences 
\(\namedddright{S^{1}}{}{P^{2}(p)}{\pi}{S^{2}}{\underline{p}}{S^{2}}\) 
and 
\(\namedddright{S^{2}}{f}{P^{2}(p)}{i}{L(p,q)}{g}{S^{3}}\) 
has dimension~$\leq 4$ so for any $n\geq 1$ we obtain exact sequence of groups 
\begin{equation} 
  \label{bundleseqmoore} 
  \namedddright{[S^{2},BU(n)]}{{\underline{p}}^{\ast}}{[S^{2},BU(n)]}{\pi^{\ast}} 
         {[P^{2}(p),BU(n)]}{}{[S^{1},BU(n)]} 
\end{equation} 
and 
\begin{equation} 
  \label{bundleseq} 
  \namedddright{[S^{3},BU(n)]}{g^{\ast}}{[L(p,q),BU(n)]}{i^{\ast}}{[P^{2}(p),BU(n)]} 
       {f^{\ast}}{[S^{2},BU(n)]}. 
\end{equation} 
Recall that $[S^{2},BU(n)]\cong\pi_{1}(U(n))\cong\mathbb{Z}$. 

\begin{lemma} 
   \label{bundles} 
   Let $n\geq 1$. The following hold: 
   \begin{letterlist} 
      \item there is a group isomorphism $[P^{2}(p),BU(n)]\cong\mathbb{Z}/p\mathbb{Z}$; 
      \item the map $\pi^{\ast}$ is reduction mod-$p$; 
      \item there is a group isomorphism  
               \(\namedright{[L(p,q),BU(n)]}{i^{\ast}}{[P^{2}(p),BU(n)]}\); 
      \item the map $\overline{\pi}^{\ast}$ is reduction mod-$p$. 
   \end{letterlist} 
\end{lemma} 
      
\begin{proof} 
In~(\ref{bundleseqmoore}), since $\underline{p}$ is the map of degree $p$, 
the induced map $\underline{p}^{\ast}$ is multiplication by $p$. As $\pi_{2}(BU(n))\cong\mathbb{Z}$ 
and $\pi_{1}(BU(n))\cong 0$, exactness in~(\ref{bundleseqmoore}) immediately implies 
that $[P^{2}(p),BU(n)]\cong\mathbb{Z}/p\mathbb{Z}$ and $\pi^{\ast}$ is reduction mod-$p$, 
proving parts~(a) and~(b). 

As $\pi_{3}(BU(n))\cong 0$, $[P^{2}(p),BU(n)]\cong\mathbb{Z}/p\mathbb{Z}$ and 
$[S^{2},BU(n)]\cong\mathbb{Z}$, from~(\ref{bundleseq}) we obtain an exact sequence of groups 
\[\namedddright{0}{}{[L(p,q),BU(n)]}{i^{\ast}}{\mathbb{Z}/p\mathbb{Z}}{}{\mathbb{Z}}.\] 
Any homomorphism from a finite group to $\mathbb{Z}$ is trivial so, by exactness, 
$i^{\ast}$ is an isomorphism, proving part~(c). 

Since $\pi\simeq\overline{\pi}\circ i$ by Corollay~\ref{pinchext}, part~(d) follows from  
parts~(b) and~(c).  
\end{proof} 

In general, if $X$ is a pointed $CW$-complex then the isomorphism classes of principal 
$U(n)$-bundles over $X$ are classified by the homotopy classes in $[X,BU(n)]$. 
If $P$ is such a bundle, classified by a map~$\alpha$, let $\g_{\alpha}(X)$ be its 
gauge group. This group has a classifying space $B\g_{\alpha}(X)$ and \mbox{by~\cite{G,AB}} 
there is a homotopy equivlalence $B\g_{\alpha}(X)\simeq\map_{\alpha}(X,BU(n))$, 
where $\map_{\alpha}(X,BU(n))$ is the component of the space of continuous maps 
from $X$ to $BU(n)$ that contains $\alpha$. The subgroup $\g_\alpha^*(X)$ of 
$G$-equivariant automorphisms of $P$ that pointwise fix the fibre at the basepoint 
is the pointed gauge group. There is a corresponding component 
of the pointed mapping space, $\mapstar_{\alpha}(X,BU(n))$, and a 
homotopy equivalence $B\g_\alpha^*(X)\simeq\mapstar_{\alpha}(X,BU(n))$. 
Evaluation of maps at the basepoint gives a homotopy fibration sequence 
\[\namedddright{U(n)}{\partial_{\alpha}}{\mapstar_{\alpha}(X,BU(n))}{} 
{\map_{\alpha}(X,BU(n))}{}{BU(n)}.\] 
The homotopy fibre of the connecting map $\partial_{\alpha}$ is $\g_{\alpha}(X)$. 

In our case, we have $[S^{2},BU(n)]\cong\mathbb{Z}$ and, by Lemma~\ref{bundles}, 
$[P^{2}(p),BU(n)]\cong [L(p,q),BU(n)]\cong\mathbb{Z}/p\mathbb{Z}$. Note that, for 
dimensional reasons, the principal $U(n)$-bundles over $S^{2}$, $P^{2}(p)$ and 
$L(p,q)$ are classified by the value of the first Chern class. For $\bar{k}\in\mathbb{Z}$, 
let $\gbark(S^{2})$ be the gauge group of the isomorphism class of principal $U(n)$-bundles 
over $S^{2}$ whose first Chern class is $\bar{k}$. For $k\in\mathbb{Z}/p\mathbb{Z}$, let 
$\gk(P^{2}(p))$ and $\gk(L(p,q))$ be the gauge groups of the isomorphism classes of 
principal $U(n)$-bundles over $P^{2}(p)$ and $L(p,q)$ respectively whose first Chern 
class is $k$. Lemma~\ref{bundles} implies that if $\bar{k}\equiv k\bmod p$ 
then there is a commutative diagram of fibration sequences 
\begin{equation} 
   \label{bignatdgrm} 
   \diagram 
         U(n)\rto^-{\partial_{\bar{k}}^{S}}\ddouble 
              & \mapstar_{\bar{k}}(S^{2},BU(n))\rto\dto^-{\overline{\pi}^{\ast}} 
              & \map_{\bar{k}}(S^{2},BU(n))\rto\dto^-{\overline{\pi}^{\ast}} & BU(n)\ddouble \\ 
         U(n)\rto^-{\partial^{L}_{k}}\ddouble & \mapstar_{k}(L(p,q),BU(n))\rto\dto^{i^{\ast}} 
              & \map_{k}(L(p,q),BU(n))\rto\dto^{i^{\ast}} & BU(n)\ddouble \\ 
         U(n)\rto^-{\partial^{P}_{k}} & \mapstar_{k}(P^{2}(p),BU(n))\rto & \map_{k}(P^{2}(p),BU(n))\rto 
              & BU(n).   
   \enddiagram 
\end{equation} 
The homotopy fibres of $\partial_{\bar{k}}(S)$, $\partial_{k}^{L}$ and $\partial_{k}^{P}$ are 
$\gbark(S^{2})$, $\gk(L(p,q))$ and $\gk(P^{2}(p))$ respectively. 

The goal is to find information about the gauge groups $\gk(L(p,q))$ via the 
middle homotopy fibration in~(\ref{bignatdgrm}). However, it is not so easy to 
study this fibration directly, one issue being that it is unclear whether the 
components $\mapstar_{k}(L(p,q),BU(n))$ are all homotopy equivalent. A similar 
issue appeared in work of the first author~\cite{MS} in dealing with gauge groups 
for principal $G$-bundles over $S^{3}$-bundles over $S^{4}$, where $G$ is a 
simply-connected, simple compact Lie group. The approach in that case involved 
localization, which needs to be avoided here since $P^{2}(p)$ need not be nilpotent. 
Instead, we obtain information indirectly: by~\cite{S}, information about $\gbark(S^{2})$ 
via the top fibration in~(\ref{bignatdgrm}) is known and in Section~\ref{sec:moore} we 
will determine information about $\gk(P^{2}(p))$ via the bottom fibration in~(\ref{bignatdgrm}). 
In Section~\ref{sec:ggLsplitting} a splitting result is proved that lets us use 
the information about $\gk(P^{2}(p))$ to deduce information about $\gk(L(p,q))$.

\section{The homotopy types of $\gk(P^{2}(p))$} 
\label{sec:moore} 

By Corollary~\ref{pinchext}, the pinch map 
\(\namedright{P^{2}(p)}{\pi}{S^{2}}\) 
is homotopic to the composite 
\(\nameddright{P^{2}(p)}{i}{L(p,q)}{\overline{\pi}}{S^{2}}.\) 
So from~(\ref{bignatdgrm}) we obtain a homotopy commutative diagram of fibration sequences 
\begin{equation} 
  \label{SPnat} 
  \diagram 
         U(n)\rto^-{\partial_{\bar{k}}^{S}}\ddouble & \mapstar_{\bar{k}}(S^{2},BU(n))\rto\dto^{\pi^{\ast}} 
              & \map_{\bar{k}}(S^{2},BU(n))\rto\dto^{\pi^{\ast}} & BU(n)\ddouble \\ 
         U(n)\rto^-{\partial_{k}^{P}} & \mapstar_{k}(P^{2}(p),BU(n))\rto 
              & \map_{k}(P^{2}(p),BU(n))\rto & BU(n).    
   \enddiagram 
\end{equation} 
First, we show that all the components $\mapstar_{k}(P^{2}(p),BU(n))$ are homotopy 
equivalent, and in a way that is compatible with a similar result from~\cite{S} about 
the components $\mapstar_{\bar{k}}(S^{2},BU(n))$. 
   
\begin{lemma} 
   \label{matchingequivs} 
   For $\bar{k}\in\mathbb{Z}$ and $k\in\mathbb{Z}/p\mathbb{Z}$ with $\bar{k}\equiv k\bmod{p}$, 
   there is a homotopy commutative diagram 
   \[\diagram 
         \mapstar_{\bar{k}}(S^{2},BU(n))\rto^-{\pi^{\ast}}\dto^{\simeq} 
                & \mapstar_{k}(P^{2}(p),BU(n))\dto^{\simeq} \\ 
         \mapstar_{0}(S^{2},BU(n))\rto^-{\pi^{\ast}} & \mapstar_{0}(P^{2}(p),BU(n)). 
      \enddiagram\] 
\end{lemma} 

\begin{proof} 
This was essentially proved in~\cite{S} but not stated in this form. An argument is given 
for the sake of completeness. Let 
\(\epsilon\colon\namedright{S^{2}}{}{BU(n)}\) 
be a fixed map with first Chern class $-\bar{k}$. Define 
\[\theta\colon\namedright{\mapstar_{\bar{k}}(S^{2},BU(n))}{}{\mapstar_{0}(S^{2},BU(n))}\] 
by sending a map 
\(f\colon\namedright{S^{2}}{}{BU(n)}\) 
with first Chern class $\bar{k}$ to the composite 
\[\theta(f)\colon\namedddright{S^{2}}{\sigma}{S^{2}\vee S^{2}}{f\vee\epsilon}{BU(n)\vee BU(n)} 
      {\nabla}{BU(n)}\] 
where $\sigma$ is the comultiplication on $S^{2}$ and $\nabla$ is the folding map. 
Similarly, define 
\[\phi\colon\namedright{\mapstar_{0}(S^{2},BU(n))}{}{\mapstar_{\bar{k}}(S^{2},BU(n))}\] 
by sending $g$ to $\nabla\circ(g\vee(-\epsilon))\circ\sigma$. Then $\theta$ and $\phi$ 
are continuous and the homotopy associativity of~$\sigma$ implies that $\phi\circ\theta$ 
and $\theta\circ\phi$ are homotopic to the identity maps. 

The space $P^{2}(p)$ is not a co-$H$-space. However, as $\pi$ is a homotopy cofibration 
connecting map there is a coaction 
\(\psi\colon\namedright{P^{2}(p)}{}{P^{2}(p)\vee S^{2}}\) 
which, when pinched to $P^{2}(p)$ is the identity map, and when pinched to $S^{2}$ 
is $\pi$. Further, this coaction has a homotopy associativity property: 
$(\psi\vee 1)\circ\psi\simeq 1\vee\sigma$. Define 
\[\theta'\colon\namedright{\mapstar_{k}(P^{2}(p),BU(n))}{}{\mapstar_{0}(P^{2}(p),BU(n))}\] 
by sending a map 
\(f'\colon\namedright{P^{2}(p)}{}{BU(n)}\) 
with first Chern class $k$ to the composite 
\[\theta'(f)\colon\namedddright{P^{2}(p)}{\psi}{P^{2}(p)\vee S^{2}}{f'\vee\epsilon}{BU(n)\vee BU(n)} 
      {\nabla}{BU(n)}\] 
and define 
\[\phi'\colon\namedright{\mapstar_{0}(P^{2}(p),BU(n))}{}{\mapstar_{k}(P^{2}(p),BU(n))}\] 
by sending $g$ to $\nabla\circ(g\vee(-\epsilon))\circ\psi$. Then, as before, $\theta'$ is 
a homotopy equivalence. 

Finally, the coaction $\psi$ satisfies a homotopy commutative diagram 
\[\diagram 
       P^{2}(p)\rto^-{\psi}\dto^{\pi} & P^{2}(p)\vee S^{2}\dto^{\pi\vee 1} \\ 
       S^{2}\rto^-{\sigma} & S^{2}\vee S^{2}. 
  \enddiagram\] 
This implies that $\theta$ and $\theta'$, and $\phi$ and $\phi'$, are compatible, 
implying the homotopy commutative diagram asserted by the lemma.  
\end{proof} 

Using $\partial_{\bar{k}}^{S}$ to also denote the composite 
\(\nameddright{U(n)}{\partial_{\bar{k}}^{S}}{\mapstar_{\bar{k}}(S^{2},BU(n)}{\simeq} 
     {\mapstar_{0}(S^{2},BU(n)}\), 
and similarly for $\partial_{k}^{P}$, by Lemma~\ref{matchingequivs} the left square 
in~(\ref{SPnat}) may be replaced with a homotopy commutative square 
\begin{equation} 
  \label{partialzero} 
  \diagram 
         U(n)\rto^-{\partial_{\bar{k}}^{S}}\ddouble & \mapstar_{0}(S^{2},BU(n))\dto^{\pi^{\ast}} \\ 
         U(n)\rto^-{\partial_{k}^{P}} & \mapstar_{0}(P^{2}(p),BU(n)).  
   \enddiagram 
\end{equation}     
By~(\ref{SPnat}), the homotopy fibres of $\partial_{\bar{k}}^{S}$ and $\partial_{k}^{P}$ 
are $\gbark(S^{2})$ and $\gk(P^{2}(p))$ respectively. 

We next identify certain self-homotopy equivalences of $\mapstar_{k}(P^{2}(p),BU(n))$. 
Since $P^{2}(p)$ is not a co-$H$-space it is not immediately clear that it has 
a degree $d$ map for any integer $d$. However, we may define an analogue as follows. 
The degree $p$ map on $S^{1}$ commutes with the degree $d$ map for any $d$, 
so we obtain a cofibration diagram 
\begin{equation} 
  \label{P2q} 
  \diagram 
      S^{1}\rto^-{p}\dto^{d} & S^{1}\rto\dto^{d} & P^{2}(p)\rto^-{\pi}\dto^{\underline{d}} & S^{2}\dto^{d} \\ 
      S^{1}\rto^-{p} & S^{1}\rto & P^{2}(p)\rto^-{\pi} & S^{2} 
  \enddiagram 
\end{equation} 
for some map $\underline{d}$. The cofibration diagram implies that, upon taking 
integral homology, $(\underline{d})_{\ast}$ is multiplication by $d$ on 
$H_{1}(P^{2}(p);\mathbb{Z})\cong\zmodp$. Suspending, since $\Sigma P^{n}(p)\simeq P^{n+1}(p)$, 
we see that for any $m\geq 1$ 
\(\namedright{P^{m+2}(p)}{\Sigma^{m}\underline{d}}{P^{m+2}(p)}\) 
induces multiplication by $d$ in $H_{m+1}(P^{t+2}(p);\mathbb{Z})\cong\zmodp$. 
Consequently, we obtain the following. 

\begin{lemma} 
   \label{suspd} 
   If $d$ is a unit in $\zmodp$ then for any $m\geq 1$ the map 
   \(\namedright{P^{m+2}(p)}{\Sigma^{m}\underline{d}}{P^{m+2}(p)}\) 
   is a homotopy equivalence. 
\end{lemma} 

\begin{proof} 
Since $d$ is a unit in $\zmodp$, the map $\underline{d}$ induces an isomorphism 
in $H_{m+1}(P^{m+2}(p);\mathbb{Z})$ and hence induces an isomorphism on 
$H_{\ast}(P^{m+2}(p);\mathbb{Z})$. Since $m\geq 1$, $P^{m+2}(p)$ is simply-connected, 
so Whitehead's Theorem implies that $\underline{d}$ is a homotopy equivalence. 
\end{proof} 

For any $d\in\mathbb{Z}$, the map $\underline{d}$ induces a map 
\[\underline{d}^{\ast}\colon\namedright{\mapstar_{0}(P^{2}(p),BU(n))}{}{\mapstar_{0}(P^{2}(p),BU(n))}.\] 

\begin{lemma} 
   \label{unit} 
   If $d$ is a unit in \zmodp\ then $\underline{d}^{\ast}$ is a homotopy equivalence. 
\end{lemma} 

\begin{proof} 
The first step is to show that $(\underline{d})^{\ast}$ induces an isomorphism on homotopy 
groups. Note that $\mapstar_{0}(P^{2}(p),BU(n))$ is path-connected since it is one component 
of $\mapstar(P^{2}(p),BU(n))$ and for all~$m\geq 1$ we have 
$\pi_{m}(\mapstar_{0}(P^{2}(p),BU(n))\cong\pi_{m}(\mapstar(P^{2}(p),BU(n))$. 
By the pointed Exponential Law, 
$\pi_{m}(\mapstar(P^{2}(p),BU(n))\cong\mapstar(P^{m+2}(p),BU(n))$. 
The effect of $\underline{d}^{\ast}$ on $\pi_{m}$ is therefore determined by applying 
$\mapstar(\ \ ,BU(n))$ to the map 
\(\namedright{P^{m+2}(p)}{\Sigma^{m}\underline{d}}{P^{m+2}(p)}\). 
Since $d$ is a unit in~$\zmodp$, by Lemma~\ref{suspd}, $\Sigma^{m}\underline{d}$ is 
a homotopy equivalence. Thus $\underline{d}^{\ast}$ induces an isomorphism 
on $\pi_{m}$. As this is true for all $m\geq 1$, $\underline{d}^{\ast}$ is a weak homotopy 
equivalence. 

Observe that $P^{2}(p)$ is a $CW$-complex and $BU(n)$ may be given the structure of 
a $CW$-complex. Doing so, by~\cite[Theorem 3]{Mi}, $\mapstar(P^{2}(p),BU(n)$ is also 
a $CW$-complex, and hence so is its component $\mapstar_{0}(P^{2}(p),BU(n))$. By 
Whitehead's Theorem, a weak homotopy equivalence between $CW$-complexes is a 
homotopy equivalence. 
\end{proof} 

Recall that, by elementary number theory, if $(u,n)=1$ then $u$ is a unit 
mod $n$, and if $(k,n)=(\ell,n)$ then $k\equiv u\ell\bmod{n}$ for some integer $u$ 
satisfying $(u,n)=1$.  

\begin{lemma} 
   \label{uklsphere} 
   Suppose that $(k,n)=(\ell,n)$, implying that $k\equiv u\ell\bmod{n}$ for some integer 
   satisfying $(u,n)=1$. Then there is a homotopy commutative diagram 
   \[\diagram 
        U(n)\rto^-{\partial^{S}_{\ell}}\ddouble & \Omega_{0} U(n)\dto^{u} \\ 
        U(n)\rto^-{\partial^{S}_{k}} & \Omega_{0} U(n).  
     \enddiagram\] 
\end{lemma} 

\begin{proof} 
In~\cite{Th}, refining work in~\cite{S}, it was shown that 
\(\namedright{U(n)}{\partial^{S}_{1}}{\Omega_{0} U(n)}\) 
has order~$n$. Therefore, $\partial^{S}_{1}$ generates a cyclic subgroup of 
order $n$ in the group $[U(n),\Omega_{0} U(n)]$. By~\cite{L}, 
$\partial^{S}_{\ell}\simeq\ell\circ\partial^{S}_{1}$ and 
$\partial^{S}_{k}\simeq k\circ\partial^{S}_{1}$. Therefore the hypothesis 
that $k\equiv u\ell\bmod{n}$ implies that $u\ell\circ\partial^{S}_{1}\simeq k\circ\partial^{S}_{1}$, 
that is, $u\circ\partial^{S}_{\ell}\simeq\partial^{S}_{k}$. 
\end{proof} 

\begin{lemma} 
   \label{uklMoore} 
   Suppose that $(k,n)=(\ell,n)$, implying that $k\equiv u\ell\bmod{n}$ for some integer 
   satisfying $(u,n)=1$. Suppose that $(u,p)=1$ as well. Then there is a 
   homotopy commutative diagram 
   \[\diagram 
        U(n)\rto^-{\partial^{P}_{\ell}}\ddouble & \mapstar_{0}(P^{2}(p),BU(n))\dto^{\underline{u}^{\ast}} \\ 
        U(n)\rto^-{\partial^{P}_{k}} & \mapstar_{0}(P^{2}(p),BU(n))  
     \enddiagram\] 
   where $\underline{u}^{\ast}$ is a homotopy equivalence. 
\end{lemma} 

\begin{proof} 
Consider the diagram 
\[\diagram 
     U(n)\rto^-{\partial^{S}_{\ell}}\ddouble & \Omega_{0} U(n)\rto^-{\pi^{\ast}}\dto^{u} 
            & \mapstar_{0}(P^{2}(p),BU(n))\dto^{\underline{u}^{\ast}} \\ 
     U(n)\rto^-{\partial^{S}_{k}} & \Omega_{0} U(n)\rto^-{\pi^{\ast}} & \mapstar_{0}(P^{2}(p),BU(n)).   
  \enddiagram\] 
The left square homotopy commutes by Lemma~\ref{uklsphere} and 
the right square homotopy commutes by applying $\mapstar(\ \ ,BU(n))$ to 
the right square in~(\ref{P2q}). Observe that the composite along the top row 
is $\partial^{P}_{\ell}$ and the composite along the bottom row is $\partial^{P}_{k}$. 
Thus $\underline{u}^{\ast}\circ\partial^{P}_{\ell}\simeq\partial^{P}_{k}$. Finally, 
the hypothesis $(u,p)=1$ implies that $u$ is a unit mod-$p$, so by Lemma~\ref{unit}, 
$\underline{u}^{\ast}$ is a homotopy equivalence. 
\end{proof} 

Note that the condition $(u,n)=1$ and $(u,p)=1$ is equivalent to the condition $(u,np)=1$. 

\begin{proposition} 
   \label{Mooregauge} 
   Suppose that $k\equiv u\ell\bmod{n}$  where $(u,np)=1$. Then  
   $\gk(P^{2}(p))\simeq\g_{\ell}(P^{2}(p))$. 
\end{proposition} 

\begin{proof} 
The homotopy fibres of $\partial^{P}_{k}$ and $\partial^{P}_{\ell}$ are $\gk(P^{2}(p))$ 
and $\g_{\ell}(P^{2}(p))$ respectively. Taking homotopy fibres in the homotopy commutative  
diagram in the statement of Lemma~\ref{uklMoore} gives an induced map of fibres 
\(\varphi\colon\namedright{\g_{\ell}(P^{2}(p))}{}{\gk(P^{2}(p))}\). 
The fact that $\underline{u}_{\ast}$ is a homotopy equivalence implies that~$\varphi$ 
is a homotopy equivalence. 
\end{proof} 

There is a partial converse to Proposition~\ref{Mooregauge} in a limited number of cases. 

\begin{lemma} 
   \label{gkhomotopygroups} 
   Suppose that $n=\mathfrak{p}$ where $\mathfrak{p}$ is a prime, that $\mathfrak{p}$ 
   divides $p$ and $(p!,\mathfrak{p})\neq p/\mathfrak{p}$. 
   If $\gk(P^{2}(p))\simeq\mathcal{G}_{\ell}(P^{2}(p))$ then $(k,\mathfrak{p})= (\ell,\mathfrak{p})$. 
\end{lemma} 

\begin{proof} 
To illustrate why the restrictions on $n$ and $p$ appear, they are introduced  
only when appropriate. For any $n$ and $p$, the homotopy cofibration 
\(\nameddright{P^{2}(p)}{\pi}{S^{2}}{\underline{p}}{S^{2}}\) 
induces an exact sequence 
\[\nameddright{\pi_{2n-1}(\Omega_{0} U(n))}{p}{\pi_{2n-1}(\Omega_{0} U(n))} 
       {\pi^{\ast}}{\pi_{2n-1}(\mapstar_{0}(P^{2}(p),BU(n)))}\] 
\[\hspace{3cm}\longrightarrow\namedright{\pi_{2n-2}(\Omega_{0} U(n))} 
       {p}{\pi_{2n-2}(\Omega_{0} U(n))}.\] 
By~\cite{BH} or~\cite{To}, $\pi_{2n}(U(n))\cong\mathbb{Z}/n!\mathbb{Z}$, 
and it is well known that 
$\pi_{2n-1}(U(n))\cong\mathbb{Z}$. As multiplication by $p$ on 
$\mathbb{Z}/n!\mathbb{Z}$ sends a generator~$\gamma$ to $(n!,p)\gamma$ and 
mutliplication by $p$ on $\mathbb{Z}$ is an injection, we obtain 
$\pi_{2n-1}(\mapstar_{0}(P^{2}(p),BU(n))\cong\mathbb{Z}/(n!,p)\mathbb{Z}$ 
and $\pi^{\ast}$ is reduction mod $(n!,p)$. 

Next, consider the commutative diagram 
\[\spreaddiagramcolumns{-0.5pc}\diagram 
     \pi_{2n-1}(U(n))\rto^-{(\partial_{\bar{k}}^{S})_{\ast}}\ddouble 
          & \pi_{2n-1}(\Omega_{0} U(n))\dto^{\pi^{\ast}} & & \\ 
     \pi_{2n-1}(U(n))\rto^-{(\partial_{k}^{P})_{\ast}} 
          & \pi_{2n-1}(\mapstar_{0}(P^{2}(p),BU(n)))\rto 
          & \pi_{2n-1}(B\gk(P^{2}(p)))\rto & \pi_{2n-1}(BU(n)) 
  \enddiagram\] 
induced by~(\ref{SPnat}), and note that the bottom row is exact. The fact that 
$\pi_{2n-1}(BU(n))\cong 0$ implies that $\pi_{2n-1}(B\gk(P^{2}(p)))$ is isomorphic 
to the cokernel of $(\partial_{k}^{P})_{\ast}$. We wish to identify this cokernel 
in a manner related to $(n,k)$ and then compare to the $(n,\ell)$ case. 

Sutherland~\cite{S} showed that the 
image of~$(\partial_{\bar{k}}^{S})_{\ast}$ is  generated by $(n-1)!(n,\bar{k})\,\gamma$. 
Let $\delta=\pi^{\ast}(\gamma)$. Then the image of $(\partial_{k}^{P})_{\ast}$ is generated 
by $(n-1)!(n,k)\,\delta$. Since $\delta$ generates a cylic group of order $(n!,p)$, 
the only way that $(n-1)!(n,k)\delta$ is not going to be trivial is if $(n!,p)$ has 
factors of $n$ which are not factors of $(n-1)!$ or $(n,k)$. The only way $n$ 
has factors that are not factors of $(n-1)!$ is if~$n$ is a prime $\mathfrak{p}$. So from 
now on assume that $n=\mathfrak{p}$. This leaves two cases: $(n,k)=(\mathfrak{p},k)$ 
is $1$ or $\mathfrak{p}$. 

Notice that now $(n!,p)=(\mathfrak{p}!,p)$. If $(\mathfrak{p},k)=1$ then 
$(n-1)!(n,k)\delta=(\mathfrak{p}-1)!\delta$ 
and $\delta$ generates $\mathbb{Z}/(n!,s)\mathbb{Z}=\mathbb{Z}/(p!,s)\mathbb{Z}$. 
The only way $(\mathfrak{p}-1)!\delta$ is nontrivial is if $\mathfrak{p}$ divides $(\mathfrak{p}!,p)$, 
implying that $\mathfrak{p}$ divides $p$. So from now on assume that $\mathfrak{p}$ 
divides $p$. Now the image of $(\partial^{P}_{k})_{\ast}$ is $(\mathfrak{p}-1)!\delta$ in the group 
$\mathbb{Z}/(\mathfrak{p}!,p)\mathbb{Z}$, and this has cokernel $\mathbb{Z}/(p/\mathfrak{p})\mathbb{Z}$. 
That is, $\pi_{2n-1}(B\gk(P^{2}(p)))\cong\mathbb{Z}/(p/\mathfrak{p})\mathbb{Z}$. 

If $(\mathfrak{p},k)=\mathfrak{p}$ then 
$(n-1)!(n,k)\delta=(\mathfrak{p}-1)!\mathfrak{p}\delta=\mathfrak{p}!\delta$ is trivial in 
$\mathbb{Z}/(\mathfrak{p}!,p)\mathbb{Z}$, implying that $(\partial^{P}_{k})_{\ast}$ is trivial, 
implying in turn that $\pi_{2n-1}(B\gk(P^{2}(p)))\cong\mathbb{Z}/(\mathfrak{p}!,p)\mathbb{Z}$. 

Assume that $(\mathfrak{p}!,p)\neq p/\mathfrak{p}$ so the cases $(\mathfrak{p},k)=1$ 
and $(\mathfrak{p},k)=\mathfrak{p}$ result in different groups for $\pi_{2n-1}(B\gk(P^{2}(p)))$. 
The same two options occur for $\pi_{2n-1}(B\g_{\ell}(P^{2}(p)))$. Therefore, if 
$\gk(P^{2}(p))\simeq\g_{\ell}(P^{2}(p))$ then 
$\pi_{2n-1}(B\gk(P^{2}(p)))\cong\pi_{2n-1}(B\g_{\ell}(P^{2}(p)))$, 
implying that $(\mathfrak{p},k)=(\mathfrak{p},\ell)$. 
\end{proof} 

Observe that Proposition~\ref{Mooregauge} and Lemma~\ref{gkhomotopygroups} 
both hold when $n=p=\mathfrak{p}$, noting in Proposition~\ref{Mooregauge} that 
$(u,np)=(u,\mathfrak{p}^{2})=1$ implies that $(u,\mathfrak{p})=1$. Therefore there 
is a complete classification of gauge groups in this case. 

\begin{proposition} 
   \label{completeMoore} 
   For a prime $\mathfrak{p}$, consider the gauge groups of principal $U(\mathfrak{p})$-bundles 
   over $P^{2}(\mathfrak{p})$. Then $\gk(P^{2}(\mathfrak{p}))\simeq\g_{\ell}(P^{2}(\mathfrak{p}))$ 
   if and only if $(\mathfrak{p},k)=(\mathfrak{p},\ell)$.~$\qqed$ 
\end{proposition}

\section{A homotopy decomposition for $\Omega^{2}\gk(L(p,q))$} 
\label{sec:ggLsplitting} 

In this section we prove Theorem~\ref{main}~(a). Consider the homotopy 
cofibration sequence 
\[\namedddright{S^{2}}{f}{P^{2}(p)}{i}{L(p,q)}{g}{S^{3}}.\] 

\begin{lemma} 
   \label{2suspL} 
   There is a homotopy equivalence $\Sigma^{2} L(p,q)\simeq P^{4}(p)\vee S^{5}$. 
\end{lemma} 

\begin{proof} 
Consider the homotopy cofibration 
\[\nameddright{S^{2}}{f}{P^{2}(p)}{}{L(p,q)}.\] 
In general, any closed, orientable $3$-manifold~$M$ is parallelizable so by~\cite{A} 
it has the property that its top cell splits off stably. In our case, as $L(p,q)$ is such a 
manifold, the attaching map~$f$ is stably trivial. As $f$ is in the stable range after 
two suspensions, we have $\Sigma^{2} f$ null homotopic. This implies that there 
is a homotopy equivalence $\Sigma^{2} L(p,q)\simeq P^{4}(p)\vee S^{5}$. 
\end{proof} 

Lemma~\ref{2suspL} implies that the map 
\(\namedright{\Sigma^{2} L(p,q)}{\Sigma^{2} g}{S^{5}}\) 
has a right homotopy inverse. We wish to choose this right homotopy 
inverse carefully; the next lemma does this given a condition on~$p$. 
Recall from Lemma~\ref{Lpofacts} that there is a homotopy cofibration 
\(\nameddright{S^{1}}{j}{L(p,q)}{h}{S^{2}\vee S^{3}}\) 
and $g$ is homotopic to $h$ composed with the map collapsing $S^{2}$ 
to a point. 

\begin{lemma} 
   \label{ginv} 
   Suppose that $2\,|\,p$. Then the map 
   \(\namedright{\Sigma^{2} L(p,q)}{\Sigma^{2} g}{S^{5}}\) 
   has a right homotopy inverse 
   \(r\colon\namedright{S^{5}}{}{}{\Sigma^{2} L(p,q)}\) 
   with the property that the composite 
   \(\nameddright{S^{5}}{r}{\Sigma^{2} L(p,q)}{\Sigma^{2} h}{S^{4}\vee S^{5}}\) 
   is homotopic to the inclusion of the right wedge summand. 
\end{lemma} 

\begin{proof} 
Let $F$ be the homotopy fibre of 
\(\namedright{P^{3}(p)}{\Sigma\pi}{S^{3}}\) 
and consider the homology Serre exact sequence for the homotopy fibration 
\(\nameddright{\Omega S^{3}}{}{F}{}{P^{3}(p)}\).  
This is an exact sequence 
\[\nameddright{H_{3}(\Omega S^{3})}{}{H_{3}(F)}{}{H_{3}(P^{3}(p))}\longrightarrow\cdots 
     \longrightarrow\namedright{H_{1}(P^{3}(p))}{}{0}.\] 
Since $H_{i}(\Omega S^{3})\cong 0\cong H_{i}(P^{3}(p))$ for $i\in\{1,3\}$ we obtain 
$H_{1}(F)\cong H_{3}(F)\cong 0$. Otherwise there is an exact sequence 
\(\namedddright{0}{}{H_{2}(\Omega S^{3})}{}{H_{2}(F)}{}{H_{2}(P^{3}(p))}\longrightarrow 0\). 
As the inclusion of the bottom cell 
\(\namedright{S^{2}}{}{P^{3}(p)}\) 
lifts to $F$ it must be the case that the sequence in $H_{2}$ is 
\(\namedddright{0}{}{\mathbb{Z}}{p}{\mathbb{Z}}{}{\mathbb{Z}/p\mathbb{Z}}\longrightarrow 0\), 
and so $H_{2}(F)\cong\mathbb{Z}$. Thus the $3$-skeleton of $F$ is homotopy equivalent 
to $S^{2}$. 

Now consider the homotopy pullback diagram 
\[\diagram 
     N\rdouble\dto & N\dto & \\ 
     F\rto\dto & P^{3}(p)\rto^-{\Sigma\pi}\dto & S^{3}\ddouble \\ 
     G\rto & \Sigma L(p,q)\rto^-{\Sigma\overline{\pi}} & S^{3} 
  \enddiagram\] 
that defines the spaces $G$ and $N$. Since 
\(\namedright{S^{3}}{\Sigma f}{P^{3}(p)}\) 
composes trivially to $\Sigma L(p,q)$, it lifts to a map 
\(\lambda\colon\namedright{S^{3}}{}{N}\). 
Let $\lambda'$ be the composite 
\(\nameddright{S^{3}}{\lambda}{N}{}{F}\). 
Since the $3$-skeleton of $F$ is $S^{2}$, $\lambda'$ factors as 
\(\namedright{S^{3}}{\lambda''}{S^{2}}\hookrightarrow F\) 
for some map $\lambda''$. Thus there is a homotopy commutative diagram 
\[\diagram 
       S^{3}\rto^-{\lambda}\dto^-{\lambda''} & N\dto \\ 
       S^{2}\rto & F.  
  \enddiagram\] 
Let $D$ be the homotopy cofibre of $\lambda''$. Then as the composite 
\(\nameddright{N}{}{F}{}{G}\) 
is null homotopic (it is two consecutive maps in a homotopy fibration), there 
is an extension 
\(\namedright{D}{}{G}\). 

We claim that $\Sigma\lambda''$ is null homotopic, implying that 
$\Sigma D\simeq S^{3}\vee S^{5}$. If so, let $r$ be the composite 
\(r\colon S^{5}\hookrightarrow\nameddright{\Sigma D}{}{\Sigma G}{}{\Sigma L(p,q)}\). 
Observe that as 
\(\nameddright{G}{}{L(p,q)}{\Sigma\overline{\pi}}{S^{3}}\) 
is a fibration, the composite $\Sigma^{2}\overline{\pi}\circ r$ is null homotopic. 
Further, the Blakers-Massey Theorem implies that the homotopy fibration 
\(\nameddright{N}{}{P^{3}(p)}{\Sigma i}{\Sigma L(p,q)}\) 
is the same as the homotopy cofibration 
\(\nameddright{S^{3}}{\Sigma f}{P^{3}(p)}{\Sigma i}{\Sigma L(p,q)}\) 
in dimensions~$\leq 3$, and $\lambda$ is the inclusion of the bottom cell. 
A homology argument then shows that the composite 
\(\nameddright{S^{5}}{r}{\Sigma^{2} L(p,q)}{\Sigma^{2} g}{S^{5}}\) 
is homotopic to the identity map. This proves the second assertion 
of the lemma. 

It remains to show that $\Sigma\lambda''$ is null homotopic. Suppose 
instead that $\Sigma\lambda''$ is essential. Then as $\pi_{4}(S^{3})\cong\mathbb{Z}/2\mathbb{Z}$, 
with generator $\eta$, we have $\Sigma\lambda''\simeq\eta$. By definition 
of $\lambda''$, the composite 
\(\namedddright{S^{3}}{\lambda''}{S^{2}}{}{F}{}{P^{3}(p)}\) 
is homotopic to $\Sigma f$. By Lemma~\ref{2suspL}, $\Sigma^{2} f$ is null 
homotopic, implying that there is a lift 
\[\diagram 
         & S^{4}\dto^{\Sigma\lambda''}\dldashed|>\tip_{\gamma} & \\ 
         M\rto & S^{3}\rto & \Sigma P^{3}(p) 
  \enddiagram\]
for some map $\gamma$, where $M$ is the homotopy fibre of 
\(\namedright{S^{3}}{}{\Sigma P^{3}(p)}\). 
The Serre exact sequence shows that the $4$-skeleton of $M$ is $S^{3}$, 
so $\gamma$ factors as 
\(\namedright{S^{4}}{t\cdot\eta}{S^{3}}\hookrightarrow M\) 
for some $t\in\mathbb{Z}/2\mathbb{Z}$. As $\Sigma\lambda''$ is assumed to 
be essential, it must be that $t=1$. Further, the composite 
\(S^{3}\hookrightarrow\namedright{M}{}{S^{3}}\) 
is the degree~$p$ map. Therefore $\Sigma\lambda''\simeq p\cdot\eta$. 
By hypothesis, $2\,|\, p$. Thus $p\cdot\eta\simeq\ast$, implying that 
$\Sigma\lambda''\simeq\ast$, a contradiction. Hence it must be the case 
that $\Sigma\lambda''$ is null homotopic. 
\end{proof} 

The aim is to use Lemma~\ref{ginv} to prove a splitting result; some care 
needs to be taken with components. The homotopy cofibration sequence 
\(\namedddright{S^{2}}{f}{P^{2}(p)}{}{L(p,q)}{g}{S^{3}}\) 
induces a homotopy fibration sequence 
\[\namedddright{\mapstar(S^{3},BU(n))}{g^{\ast}}{\mapstar(L(p,q),BU(n))}{}{\mapstar(P^{2}(p),BU(n))} 
      {f^{\ast}}{\mapstar(S^{2},BU(n))}.\] 
Let $\bar{f}^{\ast}$ be the restriction of $f^{\ast}$ to the $k^{th}$ component. 
Then we obtain a homotopy fibration diagram 
\[\spreaddiagramcolumns{-1pc}\diagram 
     \mapstar(S^{3},BU(n))\rto^-{\bar{g}^{\ast}}\ddouble & \mapstar_{k}(L(p,q),BU(n))\rto\dto 
           & \mapstar_{k}(P^{2}(p),BU(n)\rto^-{\bar{f}^{\ast}}\dto & \mapstar(S^{2},BU(n))\ddouble \\ 
     \mapstar(S^{3},BU(n))\rto^-{g^{\ast}} & \mapstar(L(p,q),BU(n))\rto 
           & \mapstar(P^{2}(p),BU(n))\rto^-{f^{\ast}} & \mapstar(S^{2},BU(n))  
  \enddiagram\] 
that defines the map $\bar{g}^{\ast}$. The map $r$ in Lemma~\ref{ginv} induces a map 
\[r^{\ast}\colon\namedright{\Omega^{2}\mapstar(L(p,q),BU(n))}{}{\Omega^{2}\mapstar(S^{3},BU(n))}\] 
with the property that $r^{\ast}\circ\Omega^{2} g$ is homotopic to the identity map. 
Let $\overline{r}$ be the composite 
\[\overline{r}\colon\nameddright{\Omega^{2}\mapstar_{k}(L(p,q),BU(n))}{} 
      {\Omega^{2}\mapstar(L(p,q),BU(n))}{r^{\ast}}{\Omega^{2}\mapstar(S^{3},BU(n))}.\]  
Then from the left square in the diagram above we immediately obtain the following. 

\begin{lemma} 
   \label{rinv} 
   Suppose that $2\,|\, p$. Then the composite $\overline{r}\circ\Omega^{2}\bar{g}^{\ast}$ 
   is homotopic to the identity map.~$\qqed$  
\end{lemma}  

By Lemma~\ref{rinv}, the homotopy fibration 
\[\nameddright{\Omega^{2}\mapstar(S^{3},BU(n))}{\Omega^{2}\bar{g}^{\ast}} 
     {\Omega^{2}\mapstar_{k}(L(p,q),BU(n))}{\Omega^{2} i^{\ast}}{\Omega^{2}\mapstar_{k}(P^{2}(p),BU(n))}\] 
splits to give a homotopy equivalence 
\begin{equation} 
  \label{Theta} 
  \Theta\colon\llnamedright{\Omega^{2}\mapstar_{k}(L(p,q),BU(n))}{\Omega^{2} i^{\ast}\times\overline{r}} 
      {\Omega^{2}\mapstar_{k}(P^{2}(p),BU(n))\times{\Omega^{2}\mapstar(S^{3},BU(n))}}. 
\end{equation} 
Recall that there are homotopy equivalences $B\g_k^*(L(p,q))\simeq\mapstar_k(L(p,q),BU(n))$ and $B\g_k^*(P^2(p))\simeq\mapstar_k(P^2(p),BU(n))$. Also, by Lemma \ref{matchingequivs} there are homotopy equivalences $\Omega\g_k^*(P^{2}(p))\simeq \Omega\g_0^*(P^{2}(p))$, for any $k\in\mathbb Z/p\mathbb Z$. Using these equivalences along with $\Theta$ we obtain homotopy decompositions of the loops on the pointed gauge group $\g_k^*(L(p,q))$. 

\begin{lemma}\label{pgg}
Suppose that $2\,|\, p$. Then there is a homotopy equivalence
\pushQED{\qed}
\[
\Omega\gk^*(L(p,q))\simeq \Omega\g_0^*(P^{2}(p))\times\Omega^4U(n).\qedhere
\]
 \popQED
\end{lemma} 

\begin{remark} 
The condition that $2\,|\, p$ is included since the map $r$ from 
Lemma~\ref{ginv} is used. However, if all we cared about was finding a left 
homotopy inverse to $\Omega^{2}\bar{g}^{\ast}$ then the weaker splitting in 
Lemma~\ref{2suspL} could be used and the $2\,|\, p$ condition dropped 
from Lemmas \ref{rinv} and \ref{pgg}. The point of using $r$ is to obtain more 
control over the fibration connecting map, as seen in Lemma~\ref{rboundary}. 
\end{remark} 

Next, we relate $\Theta$ to the twice looped boundary map 
$\Omega^{2}\partial^{L}_{k}$ to get information about $\Omega^{2}\gk(L(p,q))$. 

\begin{lemma} 
   \label{rboundary} 
   Suppose that $2\,|\, p$. Then the composite $\overline{r}\circ\Omega^{2}\partial^{L}_{k}$ 
   is null homotopic. 
\end{lemma}

\begin{proof} 
By definition, $\overline{\pi}$ is the composite 
\(\nameddright{L(p,q)}{h}{S^{2}\vee S^{3}}{\pi_{1}}{S^{2}}\) 
where $\pi_{1}$ collapses $S^{3}$ to a point. Since $[S^{3}, BU(n)]\cong 0$ 
the map $\pi_{1}$ induces an isomorphism 
\(\namedright{[S^{2},BU(n)]}{\pi_{1}^{\ast}}{[S^{2}\vee S^{3},BU(n)]}\), 
implying that there is a commutative diagram of evaluation fibration sequences 
\begin{equation} 
  \label{Lveedgrm} 
  \diagram 
      U(n)\rto^-{\partial^{S}_{\bar{k}}}\ddouble & \mapstar_{\bar{k}}(S^{2},BU(n))\rto\dto^{\pi_{1}^{\ast}} 
           & \map_{\bar{k}}(S^{2},BU(n))\rto^-{ev}\dto & BU(n)\ddouble \\ 
      U(n)\rto^-{\partial_{\bar{k}}^{\vee}}\ddouble & \mapstar_{\bar{k}}(S^{2}\vee S^{3},BU(n))\rto\dto^{h^{\ast}} 
           & \map_{\bar{k}}(S^{2}\vee S^{3},BU(n))\rto^-{ev}\dto & BU(n)\ddouble \\ 
      U(n)\rto^-{\partial_{k}^{L}} & \mapstar_{k}(L(p,q),BU(n))\rto 
           & \map_{k}(L(p,q),BU(n))\rto^-{ev} & BU(n). 
  \enddiagram 
\end{equation}  
where $\bar{k}\in\mathbb{Z}$ reduces to $k$ mod $p$. Writing 
\[\mapstar_{k}(S^{2}\vee S^{3},BU(n))\simeq\mapstar_{k}(S^{2},BU(n))\times\mapstar(S^{3},BU(n))\]  
the map $\pi_{1}^{\ast}$ in~(\ref{Lveedgrm}) is the inclusion into the left factor. Thus  
$\partial^{\vee}_{k}\simeq\partial^{S}_{k}\times\ast$. 

Now consider the diagram 
\[\diagram 
     \Omega^{2}\mapstar_{\bar{k}}(S^{2}\vee S^{3},BU(n))\rto^-{\Omega^{2} h^{\ast}}\dto^{=} 
         & \Omega^{2}\mapstar_{k}(L(p,q),BU(n))\dto &  \\ 
     \Omega^{2}\mapstar(S^{2}\vee S^{3},BU(n))\rto^-{\Omega^{2} h^{\ast}} 
          & \Omega^{2}\mapstar(L(p,q),BU(n))\rto^{r^{\ast}} 
          & \Omega^{2}\mapstar(S^{3},BU(n)). 
  \enddiagram\]
The left square is obtained by restricting to $k^{th}$-components. Notice that the 
composite along the upper direction around the diagram is $\overline{r}\circ\Omega^{2} h^{\ast}$. 
By Lemma~\ref{ginv}, the composite along the bottom row is homotopic to the projection. 
Thus $\overline{r}\circ\Omega^{2} h^{\ast}$ is homotopic to the projection. 

Using~(\ref{Lveedgrm}) we therefore obtain a homotopy commutative diagram 
\begin{equation} 
  \label{vLdgrm} 
  \diagram 
      \Omega^{2} U(n)\rto^-{\Omega^{2}\partial_{\bar{k}}^{\vee}}\ddouble 
           & \Omega^{2}\mapstar_{\bar{k}}(S^{2}\vee S^{3},BU(n))\drto^{\mbox{proj}}\dto^{\Omega^{2} h^{\ast}} & \\ 
      \Omega^{2} U(n)\rto^-{\Omega^{2}\partial_{k}^{L}} & \Omega^{2}\mapstar_{k}(L(p,q),BU(n))\rto^-{\overline{r}}  
           & \Omega^{2}\mapstar(S^{3},BU(n)).  
  \enddiagram 
\end{equation}      
As $\partial^{\vee}_{k}\simeq\partial^{S}_{k}\times\ast$, the homotopy commutativity 
of~(\ref{vLdgrm}) implies that $\overline{r}\circ\Omega^{2}\partial^{L}_{k}\simeq\ast$. 
\end{proof} 

\begin{theorem} 
   \label{gaugesplit} 
   Suppose that $2\,|\, p$. Then there is a homotopy equivalence 
   \[\Omega^{2}\gk(L(p,q))\simeq\Omega^{2}\gk(P^{2}(p))\times\Omega^{5} U(n).\] 
\end{theorem} 

\begin{proof} 
Combining the homotopy equivalence 
\[\llnamedright{\Omega^{2}\mapstar_{k}(L(p,q),BU(n))}{\Omega^{2} i^{\ast}\times\overline{r}} 
      {\Omega^{2}\mapstar_{k}(P^{2}(p),BU(n))\times{\Omega^{2}\mapstar(S^{3},BU(n))}}\]  
in~(\ref{Theta}) with the null homotopy in Lemma~\ref{rboundary} gives a homotopy 
commutative diagram 
\[\diagram 
     \Omega^{2} U(n)\rto^-{\Omega^{2}\partial^{P}_{k}}\ddouble 
          & \Omega^{2}\mapstar_{k}(P^{2}(p),BU(n))\dto \\ 
     \Omega^{2} U(n)\rto^-{\Omega^{2}\partial^{L}_{k}} 
          & \Omega^{2}\mapstar_{k}(L(p,q),BU(n)). 
    \enddiagram\] 
The asserted homotopy equivalence now follows, noting 
that $\Omega^{2}\mapstar(S^{3},BU(n))=\Omega^{4} U(n)$. 
\end{proof} 

Finally, we assemble our results to prove Theorem~\ref{main}. 

\begin{proof}[Proof of Theorem~\ref{main}] 
Part~(a) is Theorem~\ref{gaugesplit} and part~(b) is Proposition~\ref{Mooregauge}. 
\end{proof} 

\begin{corollary} 
   \label{ggLtypes} 
   Suppose that $k\equiv u\ell\bmod{n}$ where $(u,np)=1$ and that $2\,|\, p$. 
   Then there is a homotopy equivalence $\Omega^{2}\gk(L(p,q))\simeq\Omega^{2}\g_{\ell}(L(p,q))$.~$\qqed$ 
\end{corollary} 

For example, take $p=2$. Then if $(k,2n)=(\ell,2n)$ there is a homotopy equivalence 
$\Omega^{2}\gk(L(2,t))\simeq\Omega^{2}\g_{\ell}(L(2,t))$.

\bibliographystyle{amsalpha}

\begin{thebibliography}{AB}
 
\bibitem[A]{A} M.F. Atiyah, Thom complexes, \emph{Proc. London Math. Soc.} 
   \textbf{11} (1961), 291-310. 
\bibitem[AB]{AB} M.F. Atiyah and R. Bott, The Yang-Mills equations over
   Riemann surfaces, \emph{Philos. Trans. Roy. Soc. London Ser. A}
   \textbf{308} (1983), 523-615. 
\bibitem[BH]{BH} A. Borel and F. Hirzebruch, Characteristic classes 
   and homogeneous spaces: I,II, \emph{Amer. Math. J.} \textbf{80} 
   (1958), 458-538; \textbf{81} (1959), 315-382. 
\bibitem[CS]{CS} M.C. Crabb and W.A. Sutherland, Counting homotopy 
   types of gauge groups, \emph{Proc. London Math. Soc.} \textbf{81} (2000), 
   747-768.  
\bibitem[G]{G} D.H. Gottlieb, Applications of bundle map theory, 
   \emph{Trans. Amer. Math. Soc.} \textbf{171} (1972), 23-50. 
\bibitem[L]{L} G.E. Lang, The evaluation map and $EHP$ sequences,
   \emph{Pacific J. Math.} \textbf{44} (1973), 201-210. 
\bibitem[MS]{MS} I. Membrillo Solis, Homotopy types of gauge groups related to 
   $S^{3}$-bundles over $S^{4}$, arXiv:1707.07022. 
\bibitem[Mi]{Mi} J. Milnor, On spaces having the homotopy type of a $CW$-complex, 
   \emph{Trans. Amer. Math. Soc.} \textbf{90} (1959), 272-280. 
\bibitem[S]{S} W.A. Sutherland, Function spaces related to gauge groups, 
   \emph{Proc. Roy. Soc. Edin.} \textbf{121} (1992), 185-190. 
\bibitem[Th]{Th} S.D. Theriault, Homotopy decompositions of gauge groups over 
   Riemann surfaces and applications to moduli spaces, \emph{Int. J. Math} 
   \textbf{22} (2011), 1711-1719. 
\bibitem[To]{To} H. Toda, A topological proof of theorems of Bott and  
   Hirzebruch for homotopy groups of unitary groups, \emph{Mem. Coll. Sci.   
   Univ. Kyoto Ser. A Math.} \textbf{32} (1959), 103-119. 
\end{thebibliography}

\end{document}